\documentclass[]{llncs}

\usepackage{amsmath,amssymb,comment}
\usepackage{braket}
\usepackage{hyperref}
\usepackage{tikz}
\usepackage{multirow}
\usetikzlibrary{calc,arrows,automata,decorations.pathmorphing}

\newcommand{\Z}{\mathbb{Z}}
\newcommand{\N}{\mathbb{N}}

\newcommand{\B}{\mathcal{B}}

\newcommand{\C}{\mathcal{C}}
\newcommand{\M}{\mathfrak{M}}

\newcommand{\Xs}{\mathsf{X}}
\newcommand{\Ys}{\mathsf{Y}}
\newcommand{\Zs}{\mathsf{Z}}
\newcommand{\Ss}{\mathsf{S}}

\newcommand{\north}{\mathsf{North}}
\newcommand{\south}{\mathsf{South}}
\newcommand{\east}{\mathsf{East}}
\newcommand{\west}{\mathsf{West}}

\title{Subshifts, MSO Logic, and Collapsing Hierarchies}

\author{Ilkka T\"orm\"a}
\institute{
TUCS -- Turku Center for Computer Science \\
University of Turku, Finland \\
\email{iatorm{@}utu.fi}}

\begin{document}

\maketitle

\begin{abstract}
We use monadic second-order logic to define two-dimensional subshifts, or sets of colorings of the infinite plane. We present a natural family of quantifier alternation hierarchies, and show that they all collapse to the third level. In particular, this solves an open problem of [Jeandel \& Theyssier 2013]. The results are in stark contrast with picture languages, where such hierarchies are usually infinite.
\end{abstract}

\keywords{subshift, MSO logic, quantifier alternation}

\section{Introduction}

A two-dimensional subshift is a 
set of colorings of the infinite plane with finitely many colors. 
Concrete examples are given by sets of \emph{Wang tiles}, or squares with colored edges, introduced by Wang in \cite{Wa61}. The associated \emph{tiling system}
consists of all tilings of the plane 
where 
overlapping edges have the same color. 
The initial motivation for Wang tiles was to use a possible algorithm for 
the infinite tiling problem to recognize tautologies in 
first-order logic. 
The tiling problem was 
proved undecidable by Berger \cite{Be66}, and more undecidability results for tiling systems followed. 
More recently, 
strong connections between multidimensional subshifts and computability theory have been found. For example, 
it was 
shown in \cite{DuRoSh12}, \cite{AuSa13} that every vertically constant co-RE subshift can be implemented as a letter-to-letter projection of a tiling system. 
The topological entropies of tiling systems were characterized in \cite{HoMe10} as the right recursively enumerable nonnegative reals. 
It seems that every conceivable behavior occurs in the class of (projections of) tiling systems, 
if there is no obvious geometric or computational obstruction.

In this article, we follow the approach of \cite{JeTh09,JeTh13} 
and define two-dimensional subshifts by monadic second-order (MSO) logical formulas. 
We show that certain hierarchies obtained by counting quantifier alternations are finite, solving an open problem posed in \cite{JeTh13}. Classes of 
finite structures defined by MSO formulas have been studied extensively. Examples include finite words, trees, grids and graphs; see \cite{MaSc08} and references therein. For words and trees, MSO formulas define exactly the regular languages, and the quantifier alternation hierarchy collapses to the second level. On the other hand, the analogous hierarchy of picture languages was shown to be infinite in \cite{MaTh97} and strict in \cite{Sc97}. Although subshifts behave more like sets of words or trees than picture languages in this sense, the reasons are different: MSO-definable languages are regular because the geometry is so simple, 
while the subshift hierarchy collapses since we can simulate arbitrary computation already on the third level. The concept of constructing subshifts by quantifying over infinite configurations 
has 
also been studied 
in \cite{LoMaPa13} under the name of \emph{multi-choice shift spaces}, and in \cite{To14} under the more general framework of \emph{quantifier extensions}. Both 
formalisms are subsumed by MSO logic.

\section{Preliminary Definitions}

\subsection{Patterns and Subshifts}

Fix a finite \emph{alphabet} $A$. A \emph{pattern} is a map $P : D \to A$ from 
an arbitrary \emph{domain} $D = D(P) \subset \Z^2$ to $A$. A pattern with domain $\Z^2$ is a \emph{configuration}, and the set $A^{\Z^2}$ of all configurations 
is 
the \emph{full shift over $A$}. The set of finite patterns over $A$ is denoted by $A^{**}$, and those with domain $D \subset \Z^2$ by $A^D$. The restriction of a pattern $P$ to a smaller domain $E \subset D(P)$ is denoted $P|_E$. A pattern $P$ \emph{occurs at $\vec v \in \Z^2$} in another pattern $Q$, if we have $\vec v + \vec w \in D(Q)$ and $Q_{\vec v + \vec w} = P_{\vec w}$ for all $\vec w \in D(P)$. We denote $P \sqsubset Q$ if $P$ occurs in $Q$ at some coordinate. For a set of patterns $\Xs$, we denote $P \sqsubset \Xs$ if $P$ occurs in some element of $\Xs$.


A set of finite patterns $F \subset A^{**}$ defines a \emph{subshift} as the set of configurations $\Xs_F = \{ x \in A^{\Z^2} \;|\; \forall P \in F : P \not\sqsubset x \}$ where no pattern of $F$ occurs. If $F$ is finite, then $\Xs_F$ is \emph{of finite type}, or SFT. 
The \emph{language} of a subshift $\Xs \subset A^{\Z^2}$ is $\B(\Xs) = \{ P \in A^{**} \;|\; P \sqsubset \Xs \}$. For a finite $D \subset \Z^2$, we denote $\B_D(\Xs) = \B(\Xs) \cap A^D$. For $\vec v \in \Z^2$, we denote by $\sigma^{\vec v} : A^{\Z^2} \to A^{\Z^2}$ the \emph{shift by $\vec v$}, defined by $\sigma^{\vec v}(x)_{\vec w} = x_{\vec w + \vec v}$ for all $x \in A^{\Z^2}$ and $\vec w \in \Z^2$. Subshift are invariant under the shift maps.

A \emph{block map} is a function $f : \Xs \to \Ys$ between two subshifts $\Xs \subset A^{\Z^2}$ and $\Ys \subset B^{\Z^2}$ defined by a finite \emph{neighborhood} $D \subset \Z^2$ and a \emph{local function} $F : \B_D(\Xs) \to B$ which is applied to every coordinate synchronously: $f(x)_{\vec v} = F(x|_{D + \vec v})$ for all $x \in \Xs$ and $\vec v \in \Z^2$. 
The image of an SFT under a block map 
is a \emph{sofic shift}.

\begin{example}
Let $A = \{0,1\}$, and let $F \subset A^{**}$ be the set of patterns where $1$ occurs twice. Then $\Xs_F \subset A^{\Z^2}$ is the set of configurations containing at most one letter $1$. This subshift is sometimes called the \emph{sunny side up shift}, and it is sofic.

A famous example of an SFT is the \emph{two-dimensional golden mean shift} on the same alphabet, defined by the forbidden patterns $\begin{smallmatrix} 1 & 1 \end{smallmatrix}$ and $\begin{smallmatrix} 1 \\ 1 \end{smallmatrix}$. In its configurations, no two letters $1$ can be adjacent, but there are no other restrictions.
\end{example}

\subsection{Logical Formulas}

We continue the line of research of \cite{JeTh09,JeTh13}, and define subshifts by monadic second-order (MSO) formulas. We now introduce the terminology used 
in these articles, and then expand upon it. A \emph{structure} is a tuple $\M = (U, \tau)$, where $U$ is an \emph{underlying set}, and $\tau$ a \emph{signature} consisting of functions $f : U^n \to U$ and relations $r \subset U^n$ of different \emph{arities} $n \in \N$. A configuration $x \in A^{\Z^2}$ defines a structure $\M_x = (\Z^2, \tau_A)$, whose signature $\tau_A$ contains the following objects:
\begin{itemize}
\item Four unary functions, named $\north$, $\south$, $\east$ and $\west$, and called \emph{adjacency functions} in this article. They are interpreted in the structure $\M_x$ as $\north^{\M_x}((a,b)) = (a,b+1)$, $\east^{\M_x}((a,b)) = (a+1,b)$ and so on for $a, b \in \Z$.
\item For each symbol $a \in A$, a unary \emph{symbol predicate} $P_a$. It is interpreted as $P_a^{\M_x}(\vec v)$ for $\vec v \in \Z^2$ being true if and only if $x_{\vec v} = a$.
\end{itemize}

The MSO formulas that we use are defined with the signature $\tau_A$ as follows.
\begin{itemize}
\item A \emph{term (of depth $k \in \N$)} is a chain of $k$ nested applications of the adjacency functions to a first-order variable.
\item An \emph{atomic formula} is either $t = t'$ or $P(t)$, where $t$ and $t'$ are terms and $P$ is either a symbol predicate or a second-order variable.
\item A \emph{formula} is either an atomic formula, or an application of a logical connective ($\wedge, \vee, \neg, \ldots$) or first- or second-order quantification to other formulas.
\end{itemize}
The \emph{radius} of a formula is the maximal depth of a term in it. 
First-order variables (usually denoted $\vec n_1, \ldots, \vec n_\ell$) hold elements of $\Z^2$, and second-order variables hold subsets of $\Z^2$. 
Formulas without second-order variables are \emph{first-order}. 

Let $\phi$ be a closed MSO formula, and let $D \subset \Z^2$. 
A configuration $x \in A^{\Z^2}$ is a \emph{$D$-model} for $\phi$, denoted $x \models_D \phi$, if $\phi$ is true in the structure $\M_x$ when the quantification of the first-order variables in $\phi$ is restricted to 
$D$. If $D = \Z^2$, 
then we denote $x \models \phi$ and say that $x$ \emph{models} $\phi$. 
We define a set 
of configurations
$\Xs_\phi = \{ x \in A^{\Z^2} \;|\; x \models \phi \}$, 
which is always shift-invariant, but may not be a subshift. A subshift is \emph{MSO-definable} if it equals $\Xs_\phi$ for some MSO formula $\phi$.

As we find it more intuitive 
to quantify over configurations than 
subsets of $\Z^2$, and we later wish to quantify over the configurations of specific subshifts, we introduce the following definitions.
\begin{itemize}
\item The notations $\forall X[\Xs]$ and $\exists X[\Xs]$ (read \emph{for all (or exists) $X$ in $\Xs$}) define a new \emph{configuration variable} $X$, which represents a configuration of a subshift $\Xs \subset B^{\Z^2}$ over a new alphabet $B$.
\item For $X[\Xs]$ quantified as above, $b \in B$ and a term $t$, the notation $X_t = b$ defines an atomic formula that is true if and only if the configuration represented by $X$ has the letter $b$ at the coordinate represented by $t$.
\end{itemize}
MSO formulas with configuration variables instead of ordinary second-order variables are called \emph{extended MSO formulas}, and the relation $\models$ is extended to them. 
We state without proof that 
if the subshifts occurring in an extended MSO formula $\phi$ are MSO-definable,
then so is $\Xs_\phi$. Conversely, we can convert an MSO formula to an extended MSO formula by replacing every second-order variable with a configuration variable over the binary full shift. 
Unless stated otherwise, by second-order variables (usually denoted $X_1, \ldots, X_n$) we mean configuration variables, and by MSO formulas we mean extended MSO formulas.

\begin{example}
The two-dimensional golden mean shift is defined by the formula
\[ \forall \vec n \big( P_1(\vec n) \Longrightarrow \big( P_0(\north(\vec n)) \wedge P_0(\east(\vec n)) \big) \big). \]
Also, the sunny side up shift is defined by the formula
\[ \forall \vec m \forall \vec n \big( P_1(\vec n) \Longrightarrow ( P_0(\vec m) \vee \vec m = \vec n ) \big). \]
Another way to define the sunny side up shift is to use a second-order quantifier:
\[ \arraycolsep=0pt
\begin{array}{rl}
\exists U \forall & \vec n \big( U(\vec n) \Longleftrightarrow \big( U(\north(\vec n)) \wedge U(\west(\vec n))  \big) \big) \\
& \wedge \big( P_1(\vec n) \Longrightarrow \big( U(\vec n)  \wedge \neg U(\south(\vec n)) \wedge \neg U(\east(\vec n)) \big) \big).
\end{array} \]
We can 
produce an equivalent extended MSO formula, 
as per the above remark:
\[ \arraycolsep=0pt
\begin{array}{rl}
\exists X[\{0,1\}^{\Z^2}] \forall & \vec n \big( X_{\vec n} = 1 \Longleftrightarrow ( X_{\north(\vec n)} = 1 \wedge X_{\west(\vec n)} = 1 ) \big) \\
& \wedge \big( P_1(\vec n) \Longrightarrow (X_{\vec n} = 1 \wedge X_{\south(\vec n)} = 0 \wedge X_{\east(\vec n)} = 0) \big).
\end{array} \]
\end{example}

\subsection{Computability Theory}

We 
recall the \emph{arithmetical hierarchy}, a classical reference for which is \cite{Sa90}. A first-order arithmetical formula over $\N$ is $\Pi^0_0$ (equivalently, $\Sigma^0_0$), if it only contains bounded quantifiers (of the form $\forall n \leq k$ or $\exists n \leq k$). 
The formula is $\Pi^0_{k+1}$ ($\Sigma^0_{k+1}$) if it is of the form $\forall n_1 \cdots \forall n_\ell \phi$ ($\exists n_1 \cdots \exists n_\ell \phi$) where $\phi$ is $\Sigma^0_k$ ($\Pi^0_k$, respectively). Every such formula is equivalent to a $\Pi^0_k$ or $\Sigma^0_k$ one, and if it defines a subset of $\N$, that set is given the same classification. Completeness and hardness in the classes are defined using Turing reductions. For all $k \in \N$, the class $\Delta^0_{k+1} = \Pi^0_{k+1} \cap \Sigma^0_{k+1}$ contains exactly the languages decidable by Turing machines with $\Pi^0_k$ oracles. Also, $\Sigma^0_1$ is the class of recursively enumerable subsets of $\N$.

When classifying subsets of countable sets other than $\N$, we assume they are in some natural and computable bijection with $\N$. For example, a co-recursively enumerable set of finite patterns is $\Pi^0_1$. A subshift $\Xs$ is given the same classification as its language $\B(\Xs)$. If $\Xs$ is $\Pi^0_k$ for some $k \in \N$, then it can be defined by a $\Sigma^0_k$ set of forbidden patterns (the complement of $\B(\Xs)$), and a subshift defined by such a set is always $\Pi^0_{k+1}$. In particular, SFTs and sofic shifts are $\Pi^0_1$.

\begin{remark}
We use several hierarchies of subshifts obtained by counting quantifier alternations in different kinds of formulas, and the notation for them can be confusing. In general, classes defined by computability conditions (the arithmetical hierarchy) are denoted by $\Pi$ and $\Sigma$, while classes defined by MSO formulas via the modeling relation are denoted by $\bar \Pi$ and $\bar \Sigma$.
\end{remark}

\section{Hierarchies of MSO-Definable Subshifts}

In this section, we recall the definition of a hierarchy of subshift classes defined in \cite{JeTh09,JeTh13}, and then generalize it. We also state some general lemmas.

\begin{definition}
Let $C$ be a class of subshifts. An MSO formula $\psi$ is \emph{over $C$ with universal first-order quantifiers}, or $C$-u-MSO for short, if it is of the form
\[ \psi = Q_1 X_1[\Xs_1] Q_2 X_2[\Xs_2] \cdots Q_n X_n[\Xs_n] \forall \vec n_1 \cdots \forall \vec n_\ell \phi, \]
where each $Q_i$ is a quantifier, $\Xs_i \in C$, and $\phi$ is quantifier-free. If there are $k$ quantifier alternations and $Q_1$ is the existential quantifier $\exists$, then $\psi$ is called $\bar \Sigma_k[C]$, and if $Q_1$ is $\forall$, then $\psi$ is $\bar \Pi_k[C]$. The set 
$\Xs_\psi$ is given the same classification. If $C$ is the singleton class containing only the binary full shift $\{0, 1\}^{\Z^2}$, then $\psi$ is called u-MSO, and we denote $\bar \Sigma_k[C] = \bar \Sigma_k$ and $\bar \Pi_k[C] = \bar \Pi_k$. The classes $\bar \Sigma_k$ and $\bar \Pi_k$ for $k \in \N$ form the \emph{u-MSO hierarchy}.
\end{definition}

In \cite{JeTh13}, the u-MSO hierarchy was denoted by the letter $\C$, but we use the longer name for clarity. 
In the rest of this article, $C$ denotes an arbitrary class of subshifts, unless otherwise noted. We proceed with the following result, 
stated for u-MSO formulas in \cite{JeTh13}. We omit the proof, as it is essentially the same.

\begin{theorem}[Generalization of Theorem 13 of \cite{JeTh13}]
\label{thm:Compactness}
Let $\phi$ be a $C$-u-MSO formula over an alphabet $A$. Then for all $x \in A^{\Z^2}$, we have $x \models \phi$ if and only if $x \models_D \phi$ for every finite domain $D \subset \Z^2$.
\end{theorem}


\begin{corollary}
Every $C$-u-MSO formula $\phi$ over an alphabet $A$ defines a subshift.
\end{corollary}

\begin{proof}
Let $r \in \N$ be the radius of $\phi$. By Theorem~\ref{thm:Compactness}, we have $\Xs_\phi = \Xs_F$, where
$F = \{ x|_{D + [-r,r]^2} \;|\; D \subset \Z^2 \mbox{~finite}, x \in A^{\Z^2}, x \not\models_D \phi \}$.
\qed
\end{proof}

\begin{corollary}
\label{cor:UpperBounds}
For all $k, n \in \N$, we have $\bar \Pi_n[\Pi^0_k] \subset \Pi^0_{k+1}$. In particular, the u-MSO hierarchy only contains $\Pi^0_1$ subshifts.
\end{corollary}

\begin{proof}
Let $\phi = \forall X_1[\Xs_1] \exists X_2[\Xs_2] \ldots Q_n X_n[\Xs_n] \psi$ be a $\bar \Pi_n[\Pi^0_k]$ formula, where each $\Xs_i \subset A_i^{\Z^2}$ is a $\Pi^0_k$ subshift and $\psi$ is first-order. Then the product subshift $\prod_{i=1}^n \Xs_i$ is also $\Pi^0_k$. Let $P \in A^{**}$ be a finite pattern. Theorem~\ref{thm:Compactness}, together with a basic compactness argument, implies that $P \in \B(\Xs_\phi)$ holds if and only if for all finite domains $D(P) \subset D \subset \Z^2$, there exists a configuration $x \in A^{\Z^2}$ such that $x|_{D(P)} = P$ and $x \models_D \phi$. For a fixed $D$, denote this condition by $C_P(D)$.

We show that deciding $C_P(D)$ for given pattern $P$ and domain $D$ is $\Delta^0_{k+1}$. Denote $E = D + [-r,r]^2$, where $r \in \N$ is the radius of $\phi$, and let $L = \B_E(\prod_{i=1}^n \Xs_i)$. For a configuration $x \in A^{\Z^2}$, the condition $x \models_D \phi$ only depends on the finite pattern $x|_E \in A^E$, and is computable from it and the set $L$. Thus $C_P(D)$ is equivalent to the existence of a pattern $Q \in A^E$ such that $x|_E = Q$ implies $x \models_D \phi$ for all $x \in A^{\Z^2}$. Moreover, this can be decided by the oracle Turing machine that computes $L$ using a $\Pi^0_k$ oracle, and then goes through the finite set $A^E$, searching for such a $Q$. Thus the condition $C_P(D)$ is $\Delta^0_{k+1}$, which implies that deciding $P \in \B(\Xs_\phi)$ is $\Pi^0_{k+1}$. \qed
\end{proof}

Finally, if the final second-order quantifier of a u-MSO formula is universal, it can be dropped. This does not hold for $C$-u-MSO formulas in general. We omit the proof, as it is essentially the same as that of \cite[Lemma 7]{JeTh13}.

\begin{lemma}
\label{lem:UnivDrop}
If $k \geq 1$ is odd, then $\bar \Pi_k = \bar \Pi_{k-1}$, and if it is even, then $\bar \Sigma_k = \bar \Sigma_{k-1}$.
\end{lemma}

\begin{example}
\label{ex:MirrorNew}
Define the \emph{mirror shift} $\mathsf{M} \subset \{0,1,\#\}^{\Z^2}$ by the forbidden patterns $\begin{smallmatrix} a \\ \# \end{smallmatrix}$ and $\begin{smallmatrix} \# \\ a \end{smallmatrix}$ for $a \neq \#$, every pattern $\{\vec 0 \mapsto \#, (n,0) \mapsto \#\}$, and every pattern $\{(-n,0) \mapsto a, \vec 0 \mapsto \#, (n,0) \mapsto b\}$ for $n \in \N$ and $a \neq b$. A `typical' configuration of $\mathsf{M}$ contains one infinite column of $\#$-symbols, whose left and right sides are mirror images of each other. It is well-known that $\mathsf{M}$ is not sofic. We show that it can be implemented by an SFT-u-MSO formula $\psi = \forall X[\Xs] \forall \vec n_1 \forall \vec n_2 \forall \vec n_3 \phi$ in the class $\bar \Pi_1[\mathrm{SFT}]$. This also shows that Lemma~\ref{lem:UnivDrop} fails outside the u-MSO hierarchy.

\begin{figure}[htp]
\begin{center}
\begin{tikzpicture}[scale=.9]

\fill[black!20] (1.5,1) -- (2.25,1.75) -- (5.25,1.75) -- (6,1);
\fill[black!40] (2.25,1.75) -- (3.75,3.25) -- (5.25,1.75);

\draw[very thick] (1.5,1) -- (3.75,3.25) -- (6,1);
\draw[very thick,densely dotted] (2.25,1.75) -- (5.25,1.75);

\node[draw,fill=white,minimum size=.3cm,inner sep=0] () at (2.25,1.75) {$a$};
\node[draw,fill=white,minimum size=.3cm,inner sep=0] () at (3.75,3.25) {$b$};
\node[draw,fill=white,minimum size=.3cm,inner sep=0] () at (5.25,1.75) {$c$};

\draw[step=.5] (.5,.9999) grid (7,4);

\end{tikzpicture}
\end{center}
\caption{A pattern of $\Xs$ in Example~\ref{ex:MirrorNew}, containing its entire alphabet.}
\label{fig:Mirror}
\end{figure}
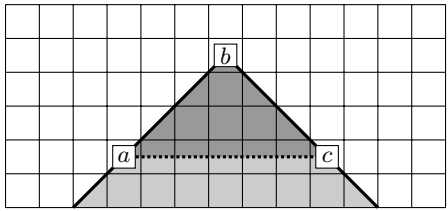

Let $\Xs$ be the SFT whose alphabet is seen in Figure~\ref{fig:Mirror}, defined by the obvious $2 \times 2$ forbidden patterns. 
Define the formula $\phi$ as $\phi_1 \wedge (\phi_2 \Longrightarrow \phi_3)$, where
\begin{align*}
\phi_1 & = P_\#(\vec n_2) \Longleftrightarrow P_\#(\north(\vec n_2)) \\
\phi_2 & = X_{\vec n_1} = a \wedge X_{\vec n_2} = b \wedge X_{\vec n_3} = c \wedge P_\#(\vec n_2) \\
\phi_3 & = \neg P_\#(\vec n_1) \wedge \neg P_\#(\vec n_3) \wedge (P_0(\vec n_1) \Longleftrightarrow P_0(\vec n_3))
\end{align*}
It is easy to see that the subshift $\Xs_\psi$ is exactly $\mathsf{M}$, with $\psi$ defined as above.
\end{example}

\section{The u-MSO Hierarchy}

The u-MSO hierarchy is 
a quite natural hierarchy of MSO-definable subshifts. Namely, the lack of existential first-order quantification makes it easy to prove that every u-MSO formula 
defines a subshift, and quantifier alternations give rise to interesting hierarchies in many contexts. 
The following 
is already known.

\begin{theorem}[\cite{JeTh13}]
The class of subshifts defined by formulas of the form $\forall \vec n \phi$, where $\phi$ is first-order, is exactly the class of SFTs. The class $\bar \Pi_0 = \bar \Sigma_0$ consists of the \emph{threshold counting shifts}, which are obtained from subshifts of the form
$\{ x \in A^{\Z^2} \;|\; \mbox{$P$ occurs in $x$ at most $n$ times} \}$
for $P \in A^{**}$ and $n \in \N$ using finite unions and intersections. Finally, the class $\bar \Sigma_1$ consists of exactly the sofic shifts.
\end{theorem}

We 
show that the 
hierarchy collapses to the third level, which consists of exactly the $\Pi^0_1$ subshifts. This gives negative answers to the 
questions posed in \cite{JeTh13} of whether the hierarchy is infinite, and whether it only contains sofic shifts.

\begin{theorem}
\label{thm:CHierarchy}
For all $n \geq 2$ we have $\Pi^0_1 = \bar \Pi_n$.
\end{theorem}

\begin{proof}
As we have $\bar \Pi_n \subset \Pi^0_1$ by Corollary~\ref{cor:UpperBounds}, and clearly $\bar \Pi_n \subset \bar \Pi_{n+1}$ also holds, 
it suffices to prove $\Pi^0_1 \subset \bar \Pi_2$. Let thus $\Xs \subset A^{\Z^2}$ be a $\Pi^0_1$ subshift. We construct an MSO formula of the form $\phi = \forall Y[B^{\Z^2}] \exists Z[C^{\Z^2}] \forall \vec n \psi(\vec n, Y, Z)$ such that $\Xs_\phi = \Xs$.

The main idea is the following. We use the universally quantified configuration $Y$ to specify a finite square $R \subset \Z^2$ and a word $w \in A^*$, which may or may not encode the pattern $x_R$ of a configuration $x \in A^{\Z^2}$. The existentially quantified $Z$ enforces that either $w$ does not correctly encode $x_R$, of that it encodes \emph{some} pattern of $\B(\Xs)$. As $R$ and $w$ are arbitrary and universally quantified, this guarantees $x \in \Xs$. The main difficulty is that $Y$ comes from a full shift, so we have no control over it; there may be infinitely many squares, or none at all.

First, we define an auxiliary SFT $\Ys \subset B^{\Z^2}$, whose configurations contain the aforementioned squares. The alphabet $B$ consists of the tiles seen in Figure~\ref{fig:FirstLayer}, where every $w_i$ ranges over $A$, and 
it is defined by the set $F_\Ys$ of $2 \times 2$ forbidden patterns where some colors or lines of neighboring tiles do not match.
A configuration of $\Ys$ contains at most one maximal pattern colored with the lightest gray in Figure~\ref{fig:FirstLayer}, and if it is finite, its domain is a square. We call this domain the \emph{input square}, and the word $w \in A^*$ that lies above it is called the \emph{input word}.

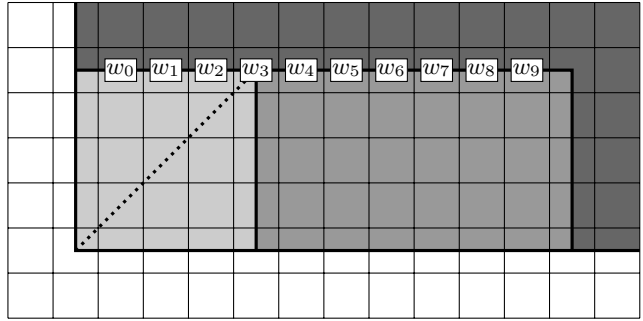
\begin{figure}[htp]
\begin{center}
\begin{tikzpicture}[scale=.6]

\fill[black!20] (1.5,1.5) rectangle (5.5,5.5);
\fill[black!40] (5.5,1.5) rectangle (12.5,5.5);
\fill[black!60] (1.5,7) -- (1.5,5.5) -- (12.5,5.5) -- (12.5,1.5) -- (14,1.5) -- (14,7);
\draw[very thick] (1.5,7) -- (1.5,1.5) -- (14,1.5);
\draw[very thick] (1.5,5.5) -- (12.5,5.5) -- (12.5,1.5);
\draw[very thick] (5.5,5.5) -- (5.5,1.5);
\draw[very thick, dotted] (1.5,1.5) -- (5.5,5.5);

\foreach \x in {0,1,...,9}{
  \draw[fill=white] (2.15+\x,5.25) rectangle ++(.7,.5);
  \node () at (2.5+\x,5.5) {$w_{\x}$};
}
\draw (0,0) grid (14,7);

\end{tikzpicture}
\end{center}
\caption{A pattern of $\Ys$. In this example, the input word $w \in A^*$ is of length $10$.}
\label{fig:FirstLayer}
\end{figure}

We now define another SFT $\Ss$, this time on the alphabet $A \times B \times C$. The alphabet $C$ is more complex than $B$, and we specify it in the course of the construction. The idea is to simulate a computation in the third layer to ensure that if the second layer contains a valid configuration of $\Ys$ and the input word encodes the contents of the input square in the first layer, then that square pattern is in $\B(\Xs)$. We also need to ensure that a valid configuration exists even if the encoding is incorrect, or if second layer is not in $\Ys$. For this, every locally valid square pattern of $\Ys$ containing an input square will be covered by another square pattern in the third layer, inside which we perform the computations. We will force this pattern to be infinite if the second layer is a configuration of $\Ys$.

Now, we describe a configuration $(x,y,z) \in \Ss$. 
The coordinates of every $2 \times 2$ rectangle $R \subset \Z^2$ with $y|_R \in F_\Ys$ are called \emph{defects}.
A non-defect coordinate $\vec v \in \Z^2$ such that 
$y_{\vec v} = \tikz[scale=.4,baseline=1]{
\fill[black!20] (.5,.5) rectangle (1,1);
\draw (0,0) rectangle (1,1);
\draw[thick] (.5,1) -- (.5,.5) -- (1,.5);
\draw[thick, densely dotted] (.5,.5) -- (1,1);}$
 is called a \emph{seed}. Denote $C = C_1 \cup C_2$, where $C_1$ is the set of tiles depicted in Figure~\ref{fig:CTiles} (a). Their adjacency rules in $\Ss$ are analogous to those of $\Ys$. The rules of $\Ss$ also force the set of seeds to coincide with the coordinates $\vec v \in \Z^2$ such that $z_{\vec v} =\tikz[scale=.4,baseline=1]{
\fill[black!20] (.5,.5) rectangle (1,1);
\draw (0,0) rectangle (1,1);
\draw[thick] (.5,1) -- (.5,.5) -- (1,.5);}$.
These coordinates are the southwest corners of \emph{computation squares} in $z$, whose square shape is again enforced by a diagonal signal. The southwest half of a computation square is colored with letters of $C_2$. See Figure~\ref{fig:CTiles} (b) for an example of a computation square.

\begin{figure}[htp]
\begin{center}
\begin{tikzpicture}

\node () at (-.5,2) {a)};


\begin{scope}[shift={(0,0)}]
\fill[black!20] (.5,.5) rectangle (1,1);
\draw[very thick] (.5,1) -- (.5,.5) -- (1,.5);
\draw (0,0) rectangle (1,1);
\end{scope}

\begin{scope}[shift={(1.25,0)}]
\fill[black!20] (0,.5) rectangle (1,1);
\draw[very thick] (0,.5) -- (1,.5);
\draw (0,0) rectangle (1,1);
\end{scope}

\begin{scope}[shift={(2.5,0)}]
\fill[black!20] (0,.5) -- (.5,.5) -- (0,1);
\fill[black!40] (0,1) -- (.5,1) -- (.5,.5);
\draw (0,1) -- (.5,.5);
\draw[very thick] (0,.5) -- (.5,.5) -- (.5,1);
\draw (0,0) rectangle (1,1);
\end{scope}

\begin{scope}[shift={(3.75,0)}]
\fill[black!40] (0,0) rectangle (.5,1);
\draw[very thick] (.5,0) -- (.5,1);
\draw (0,0) rectangle (1,1);
\end{scope}

\begin{scope}[shift={(5,0)}]
\fill[black!40] (0,0) rectangle (.5,1);
\draw[very thick] (.5,1) -- (.5,0);
\draw[very thick,densely dotted] (.25,1) -- (.25,0);
\draw (0,0) rectangle (1,1);
\end{scope}

\begin{scope}[shift={(6.25,0)}]
\fill[black!40] (0,0) rectangle (.5,1);
\draw[very thick] (.5,1) -- (.5,0);
\draw[very thick,densely dotted,->] (.25,1) -- (.25,.5);
\draw (0,0) rectangle (1,1);
\end{scope}

\begin{scope}[shift={(7.5,0)}]
\fill[black!40] (0,0) rectangle (.5,.5);
\draw[very thick] (0,.5) -- (.5,.5) -- (.5,0);
\draw[very thick,densely dotted] (0,.25) -- (.25,.25) -- (.25,0);
\draw (0,0) rectangle (1,1);
\end{scope}

\begin{scope}[shift={(8.75,0)}]
\draw (0,0) rectangle (1,1);
\end{scope}

\begin{scope}[shift={(10,0)}]
\fill[black!20] (0,0) -- (1,0) -- (0,1);
\fill[black!40] (1,0) -- (0,1) -- (1,1);
\draw (0,1) -- (1,0);
\draw (0,0) rectangle (1,1);
\end{scope}


\begin{scope}[shift={(0,1.25)}]
\fill[black!20] (.5,0) rectangle (1,1);
\draw[very thick] (.5,0) -- (.5,1);
\draw (0,0) rectangle (1,1);
\end{scope}

\begin{scope}[shift={(1.25,1.25)}]
\fill[black!20] (.5,0) -- (.5,.5) -- (1,0);
\fill[black!40] (.5,.5) -- (1,.5) -- (1,0);
\draw (.5,.5) -- (1,0);
\draw[very thick] (.5,0) -- (.5,.5) -- (1,.5);
\draw[very thick,densely dotted] (.5,.25) -- (1,.25);
\draw (0,0) rectangle (1,1);
\end{scope}

\begin{scope}[shift={(2.5,1.25)}]
\fill[black!40] (0,0) rectangle (1,.5);
\draw[very thick] (0,.5) -- (1,.5);
\draw[very thick,densely dotted] (0,.25) -- (1,.25);
\draw (0,0) rectangle (1,1);
\end{scope}

\begin{scope}[shift={(3.75,1.25)}]
\fill[black!40] (0,0) rectangle (1,.5);
\draw[very thick] (0,.5) -- (1,.5);
\draw[very thick,densely dotted,->] (0,.25) -- (.5,.25);
\draw (0,0) rectangle (1,1);
\end{scope}

\begin{scope}[shift={(5,1.25)}]
\fill[black!40] (0,0) rectangle (1,.5);
\draw[very thick] (0,.5) -- (1,.5);
\draw (0,0) rectangle (1,1);
\end{scope}

\begin{scope}[shift={(6.25,1.25)}]
\fill[black!40] (0,0) rectangle (.5,.5);
\draw[very thick] (0,.5) -- (.5,.5) -- (.5,0);
\draw (0,0) rectangle (1,1);
\end{scope}

\begin{scope}[shift={(7.5,1.25)}]
\fill[black!40] (0,0) rectangle (.5,.5);
\draw[very thick] (0,.5) -- (.5,.5) -- (.5,0);
\draw[very thick,densely dotted,->] (0,.25) -- (.25,.25);
\draw (0,0) rectangle (1,1);
\end{scope}

\begin{scope}[shift={(8.75,1.25)}]
\draw[fill=black!20] (0,0) rectangle (1,1);
\node () at (.5,.5) {$C_2$};
\end{scope}

\begin{scope}[shift={(10,1.25)}]
\draw[fill=black!40] (0,0) rectangle (1,1);
\end{scope}

\begin{scope}[shift={(2.5,-5)},scale=.5]

\node () at (-1,8.5) {b)};

\fill[black!20] (1.5,1.5) -- (7.5,1.5) -- (1.5,7.5);
\fill[black!40] (7.5,7.5) -- (7.5,1.5) -- (1.5,7.5);

\draw[thick] (1.5,1.5) rectangle (7.5,7.5);
\draw[thick,densely dotted,->] (1.5,7.25) -- (7.25,7.25) -- (7.25,3.5);
\draw (1.5,7.5) -- (7.5,1.5);

\draw[dashed] (1.5,5.5) -- (5.5,5.5) -- (5.5,1.5);

\foreach \x/\y in {0/0,1/0,7/3,8/3,7/2,8/2,9/6,10/6,9/5,10/5,9/4,10/4,3/8,4/8}{
	\node () at (\x+.75,\y+.75) {\scriptsize D};
}

\draw (0,0) grid (12,9);

\end{scope}

\end{tikzpicture}
\end{center}
\caption{The alphabet $C$ (a) and a pattern of the third layer of $\Ss$ (b), with the elements of $C_2$ represented by the featureless light gray tiles. The dashed line represents the border of an input square on the second layer. Defects are marked with a small D.}
\label{fig:CTiles}
\end{figure}
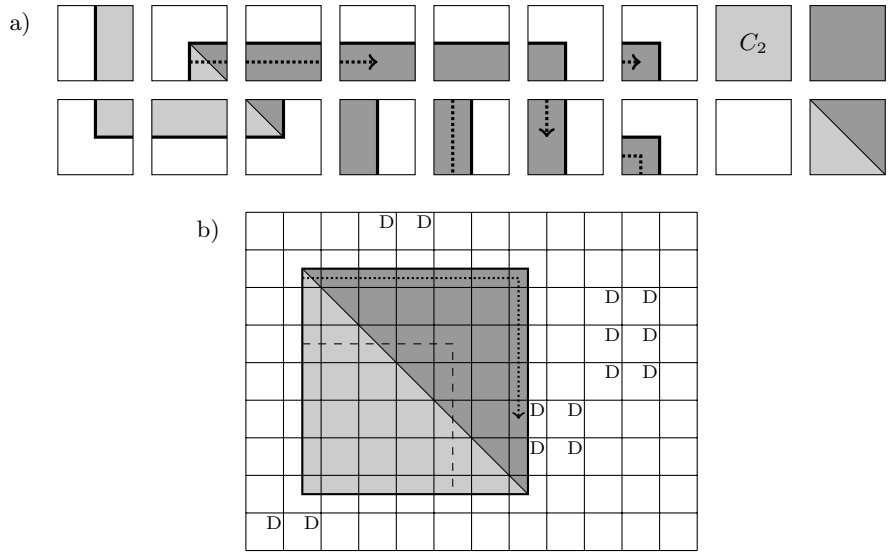

A computation square may not contain defects or coordinates $\vec v \in \Z^2$ such that $y_{\vec v} = \tikz[scale=.4,baseline=1]{\draw (0,0) rectangle (1,1);}$ except on its north or east border, and conversely, one of the borders will contain a defect. This is enforced by a signal emitted from the northwest corner of the square (the dotted line in Figure~\ref{fig:CTiles} (b)), which travels along the north and east borders, and disappears when it encounters a defect.

We now describe the set $C_2$, and for that, let $M$ be a Turing machine with input alphabet $\Sigma = A \times (A \cup \{0, 1, \#\})$ and two initial states $q_1$ and $q_2$. This machine is simulated on the southwest halves of the computation squares in a standard way, and we will fix its functionality later. The alphabet $C_2$ is shown in Figure~\ref{fig:C2Tiles}. Note that the colors and lines in $C_2$ are disjoint from those in $C_1$, even though the figures suggest otherwise. The idea is to initialize the machine $M$ with either the input word (if it correctly encodes the input square), or a proof that the encoding is incorrect, in the form of one incorrectly encoded symbol.

\begin{figure}[htp]
\begin{center}
\begin{tikzpicture}


\begin{scope}[shift={(0,0)}]
\draw[fill=black!20] (0,0) rectangle (1,1);
\end{scope}

\begin{scope}[shift={(1.5,0)}]
\fill[black!20] (0,0) rectangle (1,1);
\draw[thick] (.5,.5) -- (.5,1);
\draw[fill=white] (.25,.25) rectangle (.75,.75);
\node () at (.5,.5) {$a$};
\node[above] () at (.5,1) {$a0, q_0$};
\draw (0,0) rectangle (1,1);
\end{scope}

\begin{scope}[shift={(3,0)}]
\fill[black!20] (0,0) rectangle (1,1);
\draw[densely dotted] (.5,.5) -- (.5,1);
\draw[fill=white] (.5,.5) circle (.25);
\node () at (.5,.5) {$a$};
\node[above] () at (.5,1) {$ac$};
\draw (0,0) rectangle (1,1);
\end{scope}

\begin{scope}[shift={(4.5,0)}]
\fill[black!20] (0,0) rectangle (1,1);
\draw[densely dotted] (.5,0) -- (.5,1);
\node[above] () at (.5,1) {$s$};
\node[below] () at (.5,0) {$s$};
\draw (0,0) rectangle (1,1);
\end{scope}

\begin{scope}[shift={(6,0)}]
\fill[black!20] (0,0) rectangle (1,1);
\draw[dashed] (.5,.5) -- (.5,1);
\draw[thick] (.5,.5) -- (1,1);
\draw[fill=white] (.5,.25) -- (.25,.5) -- (.5,.75) -- (.75,.5) -- cycle;
\node () at (.5,.5) {$b$};
\node[above] () at (.5,1) {$b$};
\draw (0,0) rectangle (1,1);
\end{scope}

\begin{scope}[shift={(7.5,0)}]
\fill[black!20] (0,0) rectangle (1,1);
\draw[dashed] (.5,0) -- (.5,1);
\node[below] () at (.5,0) {$b$};
\node[above] () at (.5,1) {$b$};
\draw (0,0) rectangle (1,1);
\end{scope}

\begin{scope}[shift={(9,0)}]
\draw[fill=black!20] (0,0) rectangle (1,1);
\draw[thick] (0,0) -- (1,1);
\end{scope}


\begin{scope}[shift={(0,1.75)}]
\fill[black!20] (0,0) rectangle (1,1);
\draw[densely dotted] (.5,.5) -- (.5,1);
\draw[thick,->] (.5,0) -- (.5,.5) -- (0,.5);
\node[below] () at (.5,0) {$s,q$};
\node[above] () at (.5,1) {$t$};
\node[left] () at (0,.5) {$r$};
\draw (0,0) rectangle (1,1);
\end{scope}

\begin{scope}[shift={(1.5,1.75)}]
\fill[black!20] (0,0) rectangle (1,1);
\draw[densely dotted] (.5,0) -- (.5,.5);
\draw[thick] (.5,.5) -- (.5,1);
\draw[thick,->] (0,.5) -- (.5,.5);
\node[below] () at (.5,0) {$s$};
\node[above] () at (.5,1) {$s,q$};
\node[left] () at (0,.5) {$q$};
\draw (0,0) rectangle (1,1);
\end{scope}

\begin{scope}[shift={(3,1.75)}]
\fill[black!20] (0,0) rectangle (1,1);
\draw[thick,->] (0,.5) -- (.5,.5);
\draw[thick] (.5,.5) -- (.5,1);
\node[left] () at (0,.5) {$q$};
\node[above] () at (.5,1) {$B,q$};
\draw (0,0) rectangle (1,1);
\end{scope}

\begin{scope}[shift={(4.5,1.75)}]
\fill[black!20] (0,0) rectangle (1,1);
\draw[densely dotted] (.5,0) -- (.5,.5);
\draw[thick] (.5,.5) -- (.5,1);
\draw[thick,->] (1,.5) -- (.5,.5);
\node[below] () at (.5,0) {$s$};
\node[above] () at (.5,1) {$s,q$};
\node[right] () at (1,.5) {$q$};
\draw (0,0) rectangle (1,1);
\end{scope}

\begin{scope}[shift={(6,1.75)}]
\fill[black!20] (0,0) rectangle (1,1);
\draw[densely dotted] (.5,.5) -- (.5,1);
\draw[thick,->] (.5,0) -- (.5,.5) -- (1,.5);
\node[below] () at (.5,0) {$s,q$};
\node[above] () at (.5,1) {$t$};
\node[right] () at (1,.5) {$r$};
\draw (0,0) rectangle (1,1);
\end{scope}

\begin{scope}[shift={(7.5,1.75)}]
\fill[black!20] (0,0) rectangle (1,1);
\draw[densely dotted] (.5,.5) -- (.5,1);
\draw[dashed] (.5,0) -- (.5,.5);
\draw[fill=white] (.5,.5) circle (.25);
\node () at (.5,.5) {$a$};
\node[above] () at (.5,1) {$ab$};
\node[below] () at (.5,0) {$b$};
\draw (0,0) rectangle (1,1);
\end{scope}

\begin{scope}[shift={(9,1.75)}]
\fill[black!20] (0,0) rectangle (1,1);
\draw[densely dotted] (.5,.5) -- (.5,1);
\draw[thick] (0,0) -- (.5,.5);
\draw[fill=white] (.5,.5) circle (.25);
\node () at (.5,.5) {$a$};
\node[above] () at (.5,1) {$a\#$};
\draw (0,0) rectangle (1,1);
\end{scope}

\end{tikzpicture}
\end{center}
\caption{The sub-alphabet $C_2$. The letters $a$ and $b$ range over $A$, $c$ can be $0$ or $1$, the letter $s$ over the tape alphabet of $M$, the letter $q_0$ can be either of the initial states $q_1$ and $q_2$, and in the first (fourth) tile on the top row we require that the machine $M$ writes $t \in \Sigma$ on the tape, switches to state $r$ and steps to the left (right, respectively) when reading the letter $s \in \Sigma$ in state $q$.}
\label{fig:C2Tiles}
\end{figure}
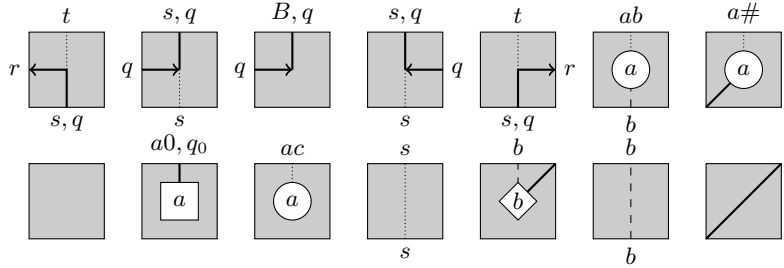

The white squares and circles of $C_2$ must be placed on the letters of the input word $w \in A^*$ of the computation square, the square on the leftmost letter and circles on the rest. The $A$-letters of these tiles must match the letters of $w$, and the second component is $1$ if the tile lies on the corner of the input square, $0$ if not, $b \in A$ in the presence of a vertical signal, and $\#$ in the presence of a diagonal signal. Such signals are sent by a white diamond tile (called a \emph{candidate error}), which can only be placed on the interior tiles of the input square, and whose letter must match the letter on the first layer $x$. Other tiles of $C_2$ simulate the machine $M$, which can never halt in a valid configuration. See Figure~\ref{fig:Computation} for a visualization. We also require that for a pattern $\begin{smallmatrix} c_2 \\ c_1 \end{smallmatrix}$ to be valid, where $c_i \in C_i$ for $i \in \{1,2\}$, the tile $c_2$ should have a gray south border with no lines. Other adjacency rules between tiles of $C_1$ and $C_2$ are explained by Figure~\ref{fig:CTiles} (a).

We now describe the machine $M$. Note first that from an input $u \in \Sigma^*$ one can deduce the input word $w \in A^*$, the height $h \in \N$ of the input square, and the positions and contents of all candidate errors. Now, when started in the state $q_1$, the machine checks that there are no candidate errors at all, that $|w| = h^2$, and that the square pattern $P \in A^{h \times h}$, defined by $P_{(i,j)} = w_{i h + j}$ for all $i, j \in [0,h-1]$, is in $\B(\Xs)$. If all this holds, $M$ runs forever (the check for $P \in \B(\Xs)$ can indeed take infinitely many steps). 
When started in $q_2$, the machine checks that there is exactly one candidate error at some position $(i,j) \in [0,h-1]^2$ of the input square containing some letter $b \in A$, and that one of $|w| \neq h^2$ or $w_{i h + j} \neq b$ holds. If this is the case, $M$ enters an infinite loop, and halts otherwise. 

\begin{figure}[htp]
\begin{center}
\begin{tikzpicture}[scale=.4]

\fill[black!20] (1.5,2.5) rectangle (20,14);

\draw[very thick] (1.5,14) -- (1.5,2.5) -- (20,2.5);
\draw[dashed] (3.5,3.75) -- (3.5,6.5);
\draw[thick] (3.5,3.5) -- (6.5,6.5);

\foreach \x in {2.5,3.5,...,12.5}{
	\draw[densely dotted] (\x,6.5) -- (\x,14);
}
\foreach \y/\xfrom/\xto in {1/0/1,2/1/2,3/2/3,4/3/2,5/2/3,6/3/4,7/4/5}{
	\draw[thick,->] (\xfrom+2.5,\y+5.5) -- (\xfrom+2.5,\y+6.5) -- (\xto/2+\xfrom/2+2.5,\y+6.5);
	\draw[thick,->] (\xto/2+\xfrom/2+2.5,\y+6.5) -- (\xto+2.5,\y+6.5);
}
\draw[fill=white] (2.25,6.25) rectangle (2.75,6.75);
\foreach \x in {3.5,4.5,...,12.5}{
	\draw[fill=white] (\x,6.5) circle (.25);
}
\draw[thick] (7.5,13.5) -- (7.5,14);

\draw[fill=white] (3.5,3.25) -- (3.25,3.5) -- (3.5,3.75) -- (3.75,3.5) -- cycle;

\draw (0,1) grid (20,14);

\end{tikzpicture}
\end{center}
\caption{An infinite computation square with an input word of length $11$ and a single candidate error.}
\label{fig:Computation}
\end{figure}
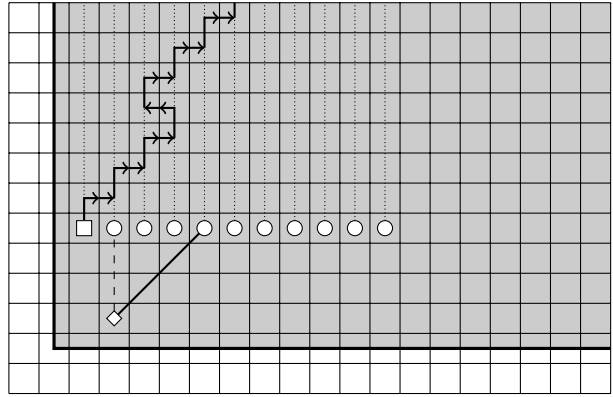

The definition of 
$\Ss$ is now complete, and it can be realized using a set 
$F$ of forbidden patterns of size $3 \times 3$. We define the quantifier-free formula $\psi(\vec n, Y, Z)$ as $\neg \bigvee_{P \in F} \psi_P$, where $\psi_P$ states that the pattern $P$ occurs at the coordinate $\vec n$. This is easily doable using the adjacency functions, color predicates and the 
variables $Y$ and $Z$. If we fix some values $y \in B^{\Z^2}$ and $z \in C^{\Z^2}$ for the variables $Y$ and $Z$, then $x \models \forall \vec n \psi(\vec n, y, z)$ holds for a given $x \in A^{\Z^2}$ if and only if $(x, y, z) \in \Ss$.

Let $x \in A^{\Z^2}$ be arbitrary. We 
need to show that $x \models \phi$ holds if and only if $x \in \Xs$. 
Suppose first that $x$ models $\phi$, and let $\vec v \in \Z^2$ and $h \geq 1$. Let $y \in \Ys$ be a configuration whose input square has interior $D = \vec v + [0, h-1]^2$, and whose input word correctly encodes the pattern $x|_D$. By assumption, there exists $z \in C^{\Z^2}$ such that $(x,y,z) \in \Ss$, so that the southwest neighbor of $\vec v$ is the southwest corner of a computation square in $z$, 
which is necessarily infinite, since no defects occur in $y$. 
In this square, $M$ runs forever, and it cannot be initialized in the state $q_2$ as the encoding of the input square is correct. Thus its computation proves that $x|_D \in \B(\Xs)$. Since $D$ was an arbitrary square domain, we have $x \in \Xs$.

Suppose then $x \in \Xs$, and let $y \in B^{\Z^2}$ be arbitrary. We construct a configuration $z \in C^{\Z^2}$ such that $(x,y,z) \in \Ss$, which proves $x \models \phi$. First, let $S \subset \Z^2$ be the set of seeds in $y$, and for each $\vec s \in S$, let $\ell(\vec s) \in \N \cup \{\infty\}$ be the height of the maximal square $D(\vec s) = \vec s + [0, \ell(\vec s)-1]^2$ whose interior contains no defects. We claim that $D(\vec s) \cap D(\vec r) = \emptyset$ holds for all $\vec s \neq \vec r \in S$. Suppose the contrary, and let $\vec v \in D(\vec s) \cap D(\vec r)$ be lexicographically minimal. Then $\vec v$ is on the south border of $D(\vec s)$ and the west border of $D(\vec r)$ (or vice versa). Since these borders contain no defects, $y_{\vec v}$ is a south border tile and a west border tile, a contradiction.

Now, we can define every $D(\vec s)$ to be a computation square in $z$. If it contains an input square and an associated input word which correctly encodes its contents, we initialize the simulated machine $M$ in the state $q_1$. Then the computation does not halt, since the input square contains a pattern of $\B(\Xs)$. Otherwise, we initialize $M$ in the state $q_2$, and choose a single candidate error from the input square such that it does not halt, and thus produces no forbidden patterns. Then $(x,y,z) \in \Ss$, completing the proof. \qed
\end{proof}

We have now characterized every level of the u-MSO hierarchy. The first level $\bar \Pi_0 = \bar \Sigma_0$ contains the threshold counting shifts and equals $\bar \Pi_1$ by Lemma~\ref{lem:UnivDrop}, the class $\bar \Sigma_1 = \bar \Sigma_2$ contains the sofic shifts, and the other levels coincide with $\Pi^0_1$.

The quantifier alternation hierarchy of MSO-definable picture languages 
was shown to be strict in \cite{Sc97}. It is slightly different from the u-MSO hierarchy, as existential first-order quantification is allowed. However, in the case of pictures we know the following. Any MSO formula $\mathcal{Q}_L \exists \vec n \mathcal{Q}_R \phi$, where $\mathcal{Q}_L$ and $\mathcal{Q}_R$ are strings of quantifiers, is equivalent to a formula of the form $\mathcal{Q}_L \exists X \mathcal{Q}_R \forall \vec n \psi$, where $\phi$ and $\psi$ are quantifier-free. See \cite[Section 4.3]{MaSc08} for more details. Thus the analogue of the u-MSO hierarchy for picture languages is infinite. The proof of the result of \cite{Sc97} relies on the fact that one can simulate computation within the pictures, and the maximal time complexity depends on the number of alternations. In the case of infinite configurations, this argument naturally falls apart.

Finally, Theorem~\ref{thm:CHierarchy} has the following corollary (which was also proved in \cite{JeTh13}).

\begin{corollary}
\label{cor:Pi01Definable}
Every $\Pi^0_1$ subshift is MSO-definable.
\end{corollary}

\section{Other $C$-u-MSO Hierarchies}

Next, we generalize Theorem~\ref{thm:CHierarchy} to hierarchies of $\Pi^0_k$-u-MSO formulas. The construction is similar to the above but easier, since we can restrict the values of the variable $Y$ to lie in a geometrically well-behaved subshift.

\begin{theorem}
\label{thm:MoreHierarchy}
For all $k \geq 1$ and $n \geq 2$ we have $\Pi^0_{k+1} = \bar \Pi_n[\Pi^0_k]$. Furthermore, $\Pi^0_2 = \bar \Pi_n[\mathrm{SFT}]$ for all $n \geq 2$.
\end{theorem}

\begin{proof}[sketch]
As in Theorem~\ref{thm:CHierarchy}, it suffices to show that for a given $\Pi^0_{k+1}$ subshift $\Xs \subset A^{\Z^2}$, there is a $\bar \Pi_2[\Pi^0_k]$ formula $\phi = \forall Y[\Ys] \exists Z[\Zs] \forall \vec n \psi$ such that $\Xs_\phi = \Xs$. In our construction, $\Ys \subset B^{\Z^2}$ is a $\Pi^0_k$ subshift and $\Zs = C^{\Z^2}$ is a full shift.

For a square pattern $P \in A^{h \times h}$, define the word $w(P) \in A^{h^2}$ by $w_{i h + j} = P_{(i,j)}$ for all $i, j \in [0, h-1]$. Let $R \subset A^* \times \N$ be a $\Pi^0_k$ predicate such that the set
\[ F = \{ P \in A^{h \times h} \;|\; h \in \N, \exists n \in \N : R(w(P), n) \} \]
satisfies $\Xs_F = \Xs$.
As in 
Theorem~\ref{thm:CHierarchy}, configurations of 
$\Ys$ may contain one input square with an associated input word. This time, the input word is of the form $w \#^n$ for some $w \in A^*$, $n \in \N$ and a new symbol $\#$. 
As $\Ys$ is $\Pi^0_k$, we can enforce that $R(w, n)$ holds, so that $w$ does \emph{not} encode any square pattern of $\Xs$. This 
can be enforced by SFT rules if $k = 1$: a simulated Turing machine checks 
$R(w,n)$ by running forever if it holds. As before, the existential layer $\Zs$ enforces that $w$ does \emph{not} correctly encode the contents of the input square in the first layer.

Let $x \in \Xs$ and $y \in \Ys$ be arbitrary. If $y$ has a finite input square $D \in \Z^2$ and input word $w \#^n$, then $w \in A^*$ cannot correctly encode the pattern $x|_D \in \B(\Xs)$, and thus a valid choice for the variable $Z$ exists. Degenerate cases of $y$ (with, say, an infinite input square) are handled as in Theorem~\ref{thm:CHierarchy}. Thus we have $x \models \phi$. Next, suppose that $x \notin \Xs$, so there is a square domain $D \subset \Z^2$ with $x|_D \notin \B(\Xs)$. Construct $y \in \Ys$ such that the input square has domain $D$, the word $w \in A^*$ correctly encodes $x|_D$, and the number $n \in \N$ of $\#$-symbols is such that $R(w, n)$ holds. For this value of $Y$, no valid choice for $Z$ exists, and thus $x \not\models \phi$. \qed
\end{proof}

Corollary~\ref{cor:Pi01Definable}, Theorem~\ref{thm:MoreHierarchy} and a simple induction argument show the following. 

\begin{corollary}
For every $k \in \N$, every $\Pi^0_k$ subshift is MSO-definable.
\end{corollary}

However, note that the converse does not hold, since one can construct an MSO-formula defining a subshift whose language is not $\Pi^0_k$ for any $k \in \N$.

\section*{Acknowledgments}

I am thankful to Emmanuel Jeandel for introducing me to \cite{JeTh09,JeTh13} and the open problems 
therein, and to Ville Salo for many fruitful discussions. 

\bibliographystyle{plain}
\bibliography{../../../bib/bib}{}

\end{document}